\documentclass[%10pt,
reqno]{amsart}
\usepackage{latexsym,amsmath,amssymb,amscd}
\usepackage[all]{xy}
\usepackage{amsthm}
\usepackage{enumerate}
\usepackage{verbatim,color}
\usepackage[pagebackref,colorlinks=true,linkcolor=blue,citecolor=red]
{hyperref} 

\usepackage[T1]{fontenc}

\makeatletter
  \newcommand\@dotsep{4.5}
  \def\@tocline#1#2#3#4#5#6#7{\relax
     \ifnum #1>\c@tocdepth % then omit
     \else
     \par \addpenalty\@secpenalty\addvspace{#2}%
     \begingroup \hyphenpenalty\@M
     \@ifempty{#4}{%
     \@tempdima\csname r@tocindent\number#1\endcsname\relax
        }{%
         \@tempdima#4\relax
           }%
      \parindent\z@ \leftskip#3\relax \advance\leftskip\@tempdima\relax
      \rightskip\@pnumwidth plus1em \parfillskip-\@pnumwidth
       #5\leavevmode\hskip-\@tempdima #6\relax
       \leaders\hbox{$\m@th
       \mkern \@dotsep mu\hbox{.}\mkern \@dotsep mu$}\hfill
       \hbox to\@pnumwidth{\@tocpagenum{#7}}\par
       \nobreak
        \endgroup
         \fi}
\makeatother 
\DeclareMathAlphabet\mathbfcal{OMS}{cmsy}{b}{n}

\newtheorem{theorem}{Theorem}

\newtheorem{lemma}[theorem]{Lemma}
\newtheorem{notation}[theorem]{Notation}

\newtheorem{proposition}[theorem]{Proposition}
\newtheorem{remark}[theorem]{Remark}

\newtheorem{conjecture}[theorem]{Conjecture}
\theoremstyle{definition}
\newtheorem{definition}[theorem]{Definition}
\newtheorem{example}[theorem]{Example}

\numberwithin{theorem}{section}
\numberwithin{equation}{section}

\newcommand{\CC}{{\mathbb C}}

\newcommand{\RR}{{\mathbb R}}
\newcommand{\TT}{{\mathbb T}}
\newcommand{\NN}{{\mathbb N}}

\newcommand{\Ac}{{\mathcal A}}

\newcommand{\Cc}{{ C}}

\newcommand{\Gc}{{\mathcal G}}

\newcommand{\Ic}{{\mathcal I}}
\newcommand{\Jc}{{\mathcal J}}

\newcommand{\Oc}{{\mathcal O}}

\newcommand{\Qc}{{\mathcal Q}}

\newcommand{\Tc}{{\mathcal T}}
\newcommand{\Vc}{{\mathcal V}}

\newcommand{\Xc}{{\mathcal X}}

\newcommand{\Wc}{{\mathcal W}}

\newcommand{\Pg}{{\mathfrak P}}
\newcommand{\Sg}{{\mathfrak S}}

\newcommand{\ag}{{\mathfrak a}}

\renewcommand{\gg}{{\mathfrak g}}
\newcommand{\hg}{{\mathfrak h}}

\newcommand{\pg}{{\mathfrak p}}

\newcommand{\Ad}{{\rm Ad}}

\newcommand{\Aut}{{\rm Aut}\,}
\newcommand{\ad}{{\rm ad}}

\newcommand{\Comm}{{\rm Comm}}

\newcommand{\de}{{\rm d}}
\newcommand{\ee}{{\rm e}}
\newcommand{\End}{{\rm End}\,}

\newcommand{\ie}{{\rm i}}

\newcommand{\Ind}{{\rm Ind}}
\newcommand{\ind}{{\rm ind}\,}
\newcommand{\Ker}{{\rm Ker}\,}

\newcommand{\LC}{{\rm LC}}

\newcommand{\Ran}{{\rm Ran}\,}

\newcommand{\Tr}{{\rm Tr}\,}

\newcommand{\Gr}{{\rm Gr}}
\newcommand{\alg}{{\rm alg}}

\newcommand{\Subquot}{{\rm SQ}}
\newcommand{\Hausd}{{\rm H}}

\newcommand{\dual}[2]{\langle #1, #2\rangle}

\newcommand{\Sbd}{{\mathfrak S}}

	\newcommand{\tto}{\rightrightarrows}

\begin{document}

\title[$C^*$-algebras of linear dynamical systems]
{On the $C^*$-algebras of linear dynamical systems}
\author{Ingrid Belti\c t\u a}  
\author{Daniel Belti\c t\u a}
\address{Institute of Mathematics ``Simion Stoilow'' 
of the Romanian Academy,   
P.O. Box 1-764, Bucharest, Romania}
\email{ingrid.beltita@gmail.com, Ingrid.Beltita@imar.ro}

\address{Institute of Mathematics ``Simion Stoilow'' 
of the Romanian Academy,   
P.O. Box 1-764, Bucharest, Romania}
\email{beltita@gmail.com, Daniel.Beltita@imar.ro}

%\thanks{}
%\date{\today; BB\_subquot-low\_2022.tex}
\date{\today . File name: \jobname.tex}
\thanks{This work was supported by a grant of the Ministry of Research, Innovation and Digitization, CNCS--UEFISCDI, project number PN-IV-P1-PCE-2023-0264, within PNCDI IV}

\begin{abstract}
We verify the conjecture on continuous-trace subquotients for  $C^*$-algebras of nilpotent linear dynamical systems, where by linear dynamical system we mean a continuous action of the additive group of real numbers  by linear maps on a finite-dimensional real vector space. 
In addition, we show that the dimension of the ambient vector space can be recovered from the corresponding $C^*$-algebra and, if the action is 
nilpotent of degree two, the corresponding group is $C^*$-rigid within the class of 1-connected nilpotent Lie groups with coadjoint orbits of dimension $\le 2$. 
\end{abstract}

\maketitle

%\tableofcontents

\section{Introduction}

In this note, by linear dynamical system we mean a continuous action of the additive group of real numbers $\RR$ by linear maps on a finite-dimensional real vector space $\Vc$. 
These actions can be given via linear operators on $\Vc$, 
namely an operator $D\in\End_\RR(\Vc)$ corresponds to the group action 
$\alpha^D\colon \RR\times\Vc\to\Vc$, $(t,v)\mapsto\ee^{tD}v$. 
A~structure of that type gives rise to  the semidirect product group $G_D:=\Vc\rtimes_{\alpha^D}\RR$. 
This is a solvable Lie group 
with its corresponding $C^*$-algebra $C^*(G_D)$. 
As shown in~\cite{BB18a}, several spectral properties of the operator $D$ are reflected in approximation properties of $C^*(G_D)$, which suggests the problem of recovering the data $(\Vc,D)$ from the $C^*$-algebra $C^*(G_D)$. 

That problem requires the study of the fine structure of $C^*(G_D)$ encoded by its subquotients. 
To this end, one needs an alternative description of that $C^*$-algebra in a framework that is more general than the group $C^*$-algebras. 
For, the subquotients of $C^*(G_D)$ fail to be group $C^*$-algebras in general, while many of them are $C^*$-algebras associated to transformation groups defined by the actions of~$\RR$ on suitable subsets of the dual vector space $\Vc^*$. 
We are thus naturally lead to working within the framework of groupoid $C^*$-algebras. 
Specifically, one has the transformation-group Lie groupoid $\Gc_D:=\RR\times\Vc^*\tto\Vc^*$, 
which in turn gives rise to the groupoid $C^*$-algebra $C^*(\Gc_D)$ that is 
$*$-isomorphic to the group $C^*$-algebra $C^*(G_D)$, 
cf. \cite{Re80}, \cite{Wi19}, and also \cite{BB18b}.  
More generally, for a solvable Lie group $G$, the structure of the cotangent groupoid of $G$,
is closely related to the fine structure of the subquotients of $C^*(G)$, as seen in \cite{BB18b}. 

This new perspective  leads to some progress  both in the above recovery problem 
and the following conjecture. 

%\begin{problem}[{\cite{LL13}}]
%\label{iso-probl}
%\normalfont
%If $G_1$ and $G_2$ are 1-connected nilpotent Lie groups and $C^*(G_1)$ is isomorphic to $C^*(G_2)$, does it follow that $G_1$ is isomorpphic to $G_2$? 
%\end{problem}

\begin{conjecture}[{\cite[Conj. 4.18]{RaRo88}}]
\label{SQ-prel}
\normalfont
If $G$ is an exponential solvable Lie group 
(e.g., a 1-connected nilpotent Lie group)
and $\Jc_1\subseteq\Jc_2$ are closed two-sided ideals of $C^*(G)$ whose quotient $\Jc_2/\Jc_1$ is a $C^*$-algebra with continuous trace, then $\Jc_2/\Jc_1$ is Morita-equivalent to a commutative $C^*$-algebra. 
\end{conjecture}

The above Conjecture~\ref{SQ-prel} was solved in the affirmative in only a few special cases: 
\begin{itemize}
	\item when $G$ is a two-step nilpotent Lie group, cf. \cite[Thm. 3.4]{LiRo96}; 
	\item when $G$ is a three-step nilpotent Lie group with one-dimensional center and generic flat coadjoint orbits, cf. \cite[Cor. 6.9]{BBL17}; 
	\item when $G$ belongs to a certain countable family of nilpotent Lie groups whose all coadjoint orbits are flat, cf. \cite[Cor. 6.14]{BBL17}.
\end{itemize}
Theorem~\ref{more4} below shows that, again for a nilpotent Lie algebra $\gg$ with a 1-codimensional abelian subalgebra, its corresponding Lie group $G$ can be added to the above list.

The fine structure of the  $C^*$ algebra of a group is central in solving the $C^*$-rigidity problem.
We refer to \cite{BB21} and \cite{BB25} for progress obtained so far on 
 the open problem of recovering a nilpotent Lie group from its corresponding $C^*$-algebra. 
We show below that, in the case of a nilpotent Lie algebra $\gg$ with a 1-codimensional abelian subalgebra, at least $\dim\gg$ and $\dim[\gg,\gg]$ can be recovered from $C^*(G)$ (Proposition~\ref{rec}). 

\subsection*{General notation and terminology} 
%\hfill\\
We denote  $\TT:=\{z\in\CC\mid \vert z\vert=1\}$ and  $\RR^\times:=\RR\setminus\{0\}$.
Throughout this paper, `1-connected' means connected and simply connected. 
Every 1-connected Lie group is denoted by an upper case Roman letter, 
and its Lie algebra is denoted by the corresponding lower case Gothic letter.

\section{Preliminaries}

\subsection*{Subquotients of $C^*$-algebras}
We recall some notation and auxiliary facts from \cite[Sect. 6]{BBL17}.

\begin{notation}
	\normalfont
	For every $C^*$-algebra $\Xc$ we denote by $[\Xc]$ its $*$-isomorphism class, and the closed two-sided ideals of $\Xc$ are called simply ideals. 
	The unitary dual space $\widehat{\Xc}$ is the set of all unitary equivalence classes of irreducible $*$-representations of $\Xc$, 
	and $\widehat{\Xc}$ is regarded as a topological space with the Fell topology. 
	
	For any $C^*$-algebra $\Ac$ let $\Subquot(\Ac)$ be the set of all $*$-isomorphism classes of subquotients of $\Ac$, 
	that is
	$$\Subquot(\Ac):=\{[\Jc_2/\Jc_1]\mid \text{$\Jc_1\subseteq\Jc_2$ are ideals of $\Ac$}\}.$$
	Let 
	\begin{align*}
	\Subquot^{\Hausd}(\Ac):=
	&\{[\Jc_2/\Jc_1]\in\Subquot(\Ac)\mid 
	\widehat{\Jc_2/\Jc_1} \text{ is a Hausdorff topological space}\}, \\
	\Subquot^{\Tr}(\Ac):=
	&\{[\Jc_2/\Jc_1]\in\Subquot(\Ac)\mid 
	\Jc_2/\Jc_1 \text{ is a continuous-trace $C^*$-algebra}\}.
	\end{align*}
	We also denote by  $\Subquot^{\Tr}_0(\Ac)$ the subset of $\Subquot^{\Tr}(\Ac)$ corresponding to the subquotients that are Morita-equivalent to commutative $C^*$-algebras. 
	Thus
	$$\Subquot^{\Tr}_0(\Ac)\subseteq\Subquot^{\Tr}(\Ac)\subseteq\Subquot^{\Hausd}(\Ac)\subseteq\Subquot(\Ac).$$
\end{notation}

With the above notation, Conjecture~\ref{SQ-prel} is equivalently stated as follows: 

\begin{conjecture}\label{SQ0}
	\normalfont
	If $\Ac$ is the $C^*$-algebra of an arbitrary 
	exponential Lie group, then $\Subquot^{\Tr}_0(\Ac)=\Subquot^{\Tr}(\Ac)$. 
\end{conjecture}

We now briefly describe the well-known parametrization of $\Subquot(\Ac)$ in terms of locally closed subsets of the dual space $\widehat{\Ac}$ of any $C^*$-algebra $\Ac$. 
First recall that a \textit{locally closed} subset of a topological space is a 
set that is the intersection of an open and a closed subset. 
In particular, the open subsets and the closed subsets are locally closed. 
For any topological space $T$ we denote by $\LC(T)$ the set of all locally closed subsets of $T$.

	For any closed two-sided ideal $\Jc$ of $\Ac$ we can regard $\widehat{\Ac/\Jc}$ and $\widehat{\Jc}$ as subsets of~$\widehat{A}$ via 
	$$\begin{aligned}
	\widehat{\Ac/\Jc}\simeq \widehat{\Ac}_{\Jc}&:=\{[\pi]\in\widehat{\Ac}\mid \Jc\subseteq\Ker\pi\}\subseteq \widehat{A}, \\
	\widehat{\Jc}\simeq \widehat{\Ac}^{\Jc}&:=\{[\pi]\in\widehat{\Ac}\mid \Jc\not\subset\Ker\pi\} \subseteq \widehat{A}. 
	\end{aligned} $$
Moreover $\Jc\mapsto \widehat{\Jc}$ is an increasing bijection between the 
		closed two-sided ideals of $\Ac$ and the open subsets of $\widehat{\Ac}$, 
		while $\Jc\mapsto \widehat{\Ac/\Jc}$ is a decreasing bijection between the 
		closed two-sided ideals of $\Ac$ and the closed subsets of $\widehat{\Ac}$.
(See for instance \cite[Props. 2.11.2, 3.2.2]{Dix64}.)

More generally, a subquotient of $\Ac$ is any $C^*$-algebra of the form $\Jc_2/\Jc_1$, 
	where $\Jc_1\subseteq\Jc_2$ are any closed two-sided ideals of $\Ac$. 
Then 
	$D:=\widehat{\Jc_1}$ is an open subset of $\widehat{\Ac}$ 
	and $\widehat{\Jc_2/\Jc_1}$ is a closed subset of the open set $\widehat{\Jc_1}$, 
	hence it is easily checked that the disjoint union $F:=\widehat{\Jc_2/\Jc_1}\cup(\widehat{\Ac}\setminus\widehat{\Jc_1})$ is 
	a closed subset of $\widehat{\Ac}$ and $F\cap D=\widehat{\Jc_2/\Jc_1}$ is locally closed. 
	The $*$-isomorphism class of the $C^*$-algebra $\Jc_2/\Jc_1$
	depends only on the locally closed set $\widehat{\Jc}_2\setminus \widehat{\Jc}_1$ 
	and the map 
	\begin{equation}\label{Psi_eq}
	\Psi_\Ac\colon \Subquot(A)\to\LC(\widehat{\Ac}),\quad [\Jc_2/\Jc_1]\mapsto \widehat{\Jc_2/\Jc_1}
	\end{equation}
	is a well-defined bijection. 
	See \cite[Lemma~7.3.5]{Ph87} and \cite[Rem. 2.4]{BBL17} for more details.
	
\subsection*{Polarizations and subordinated subalgebras} 
Let $\gg$ be any finite-dimensional real Lie algebra. 
We denote by $\Gr_{\alg}(\gg)$ the set of all subalgebras of $\gg$. 
	For every $\xi\in\gg^*$ we define 
	$$\begin{aligned}
	\Sbd(\xi)
	&:=\{\hg\in \Gr_{\alg}(\gg)\mid [\hg,\hg]\subseteq\Ker\xi\}, \\
	\Pg(\xi)
	&:=\{\hg\in \Gr_{\alg}(\gg)\mid \hg \text{ maximal element of }\Sbd(\xi)\} \\
	\end{aligned}$$
	hence $\Sbd(\xi)$ is the set of all subordinated subalgebras and $\Pg(\xi)$ is the set of all polarizations at~$\xi$. 
	
\begin{remark}\label{grcomp}
	\normalfont
	Let $\Vc$ be any finite-dimensional real vector space and 
	denote by $\Gr(\Vc)$ the compact Grassmann manifold that consists of all the linear subspaces of $\Vc$. 
	If $\lim\limits_{j\to\infty}\Wc_j=\Wc$ in $\Gr(\Vc)$, 
	then for every $w\in\Wc$ there exist $w_j\in\Wc_j$ for all $j\in\NN$ with 
	$\lim\limits_{j\to\infty}w_j=w$ in $\Vc$. 
	This can be seen by using a local chart of the smooth manifold $\Gr(\Vc)$. 
\end{remark} 

\begin{remark}\label{gralg}
	\normalfont
	As proved in \cite[1.11.9]{Dix74}, $\Gr_{\alg}(\gg)$ is a Zariski-closed subset of 
	the Grassmann manifold $\Gr(\gg)$. 
	In particular, $\Gr_{\alg}(\gg)$ is a compact topological space with 
	its topology inherited as a subset of $\Gr(\gg)$. 
\end{remark}

\begin{lemma}\label{acc_item2}
	Assume $\lim\limits_{j\to\infty}\xi_j=\xi$ in $\gg^*$ and select arbitrarily 
		$\pg_j\in\Sbd(\xi_j)$.  
		Then for any cluster point $\pg$ of the sequence $\{\pg_j\}_{j\in\NN}$ in $\Gr(\gg)$, 
		we have $\pg\in\Sbd(\xi)$. 
\end{lemma}

\begin{proof}
	By selecting a suitable subsequence 
	%and using  Remark~\ref{gralg}, 
	we may assume $\pg=\lim\limits_{j\to\infty}\pg_j$ in $\Gr(\gg)$. 
	Then $\pg\in\Gr_{\alg}(\gg)$ by Remark~\ref{gralg}. 
	Moreover, for all $x,y\in\pg$ there exist $x_j,y_j\in\pg_j$ for all $j\in\NN$ with 
	$\lim\limits_{j\to\infty}x_{j}=x$ and $\lim\limits_{j\to\infty}y_{j}=y$ 
	(by Remark~\ref{grcomp}). 
	Therefore 
	$$\dual{\xi}{[x,y]}=\lim\limits_{j\to\infty}\dual{\xi_{j}}{[x_j,y_j]}=0$$
	and thus $\pg\in\Sbd(\xi)$. 
\end{proof}

\section{Subquotients and generalized controlled boundary extensions}

\subsection*{Extensions of $C^*$-algebras and continuous-trace subquotients}

Here is a more general version of \cite[Lemma 6.5]{BBL17}.

\begin{lemma}\label{SQ5bis}
	Let $X$ be a first-countable topological space with $X= X_1 \sqcup X_2$, where 
	the following condition is satisfied: 
	\begin{itemize}%{enumerate}[(i)] 
		 \item If $\overline{x}=\{x_j\}_{j\in \NN}$ is a convergent sequence contained in $X_1$, 
		with  
		its set of limit points $L(\overline{x})$, 
		then the set $L(\overline{x})\cap X_2$ has no isolated points. 
	\end{itemize}%{enumerate}
	Then for every locally closed and Hausdorff subset $S$ of $X$ the set $S\cap X_1$ is relatively closed in $S$. 
\end{lemma}

\begin{proof}
	Let $y\in S\cap \overline{S\cap X_1}$. 
	We must prove that $y \in S\cap X_1$. 
	Assume the contrary, that is, $y \not \in S\cap X_1$,  
	hence $y\in S\cap X_2$. 
	Since $y \in   \overline{S\cap X_1}$ and  $X$ is a first-countable topological space, there exists a convergent sequence $\overline{x}=\{x_j\}_{j\in \NN}$ 
	contained in $S\cap X_1$, with  $y\in L(\overline{x})$, hence $y\in L(\overline{x})\cap S\cap X_2$. 
	
	The set $S$ is locally closed, hence there are $D\subset X$ open and $F\subset X$ closed with $S=F\cap D$. 
	Since $y\in S\subset D$, we obtain $y\in (L(\overline{x})\cap X_2)\cap D$. 
	By hypothesis, the set $L(\overline{x})\cap X_2$ has no isolated points, 
	hence there is $y'\in (L(\overline{x})\cap X_2)\cap D$ with $y'\ne y$, 
	and thus 
	$$\{y,y'\}\subseteq (L(\overline{x})\cap X_2)\cap D.$$ 
	On the other hand $\overline{x}\subseteq S\subseteq F$, where we recall that $F$ is closed, 
	and then $L(\overline{x})\subseteq F$. 
	We thus obtain that $\{y, y'\}\subseteq L(\overline{x})\cap D\cap F=L(\overline{x})\cap S$, 
	which is a contradiction with the assumption that $S$ has the Hausdorff property. 
	
	Consequently we must have $y\in S\cap X_1$, and this completes the proof.
\end{proof}

\begin{example}\label{ex12}
	\normalfont
	Here are two examples of nilpotent Lie groups $G$, 
	whose decompositions of $X:=\widehat{G}$ according to the dimensions of coadjoint orbits satisfy the condition of Lemma~\ref{SQ5bis}. 
	For the second of these examples, the hypothesis of \cite[Lemma 6.5]{BBL17} is \emph{not} satisfied. 
	\begin{enumerate}[(i)]
		\item\label{ex12_item1} Let $G$ be the $(2n+1)$-dimensional Heisenberg group. 
		If $X_1$ is its set of $2n$-dimensional coadjoint orbits and $X_2$ is its set of $0$-dimensional coadjoint orbits, 
		then $\widehat{G}=X_1\sqcup X_2$ is a decomposition as above, as it is well known. 
		\item\label{ex12_item2} Let $G$ be the $m$-dimensional thread-like group with $m\ge 3$. 
		If $X_1$ is its set of $2$-dimensional coadjoint orbits and $X_2$ is its set of $0$-dimensional coadjoint orbits, 
		then $\widehat{G}=X_1\sqcup X_2$ is again a decomposition as above. 
		In fact, 
		by the proof of \cite[Cor. 3.4]{ALS07}, if $\overline{x}=\{x_j\}_{j\in \NN}$ is a convergent sequence contained in $X_1$, 
		then either $L(\overline{x})\subseteq X_1$,  
		or for every $\chi\in L(\overline{x})\cap X_2$ there exists an affine line $A\subseteq X_1\simeq\gg/[\gg,\gg]$ with 
		$\chi\in A\subseteq L(\overline{x})\cap X_2$, hence in this case the set $L(\overline{x})\cap X_2$ has no isolated points. 
	\end{enumerate}
\end{example}

\begin{definition}\label{SQ6bis} 
	\normalfont
	The short exact sequence of separable $C^*$-algebras 
	$$
	0\to\Ac_1\to\Ac\mathop{\to}\limits^p\Ac_2\to0
	$$
	is called a \emph{generalized controlled boundary extension} 
	if,  denoting $X=\widehat{\Ac}$, $X_1=\widehat{\Ac_1}$, $X_2=\widehat{\Ac_2}$, the topological space
	$X= X_1\sqcup X_2$ satisfies the conditions in Lemma~\ref{SQ5bis}. 
\end{definition}

Now we obtain the following theorem by the method of proof of \cite[Thm. 6.7]{BBL17}, 
using the above Lemma~\ref{SQ5bis} instead of \cite[Lemma 6.5]{BBL17}. 

\begin{theorem}\label{SQ7bis}
	Let 
	$0\to\Ac_1 \to\Ac\mathop{\to}\limits^p\Ac_2\to0$
	be a generalized controlled boundary extension, 
	and fix any subset $\Qc\subseteq\Subquot(\Ac)$ which is hereditary, 
	in the sense that 
	if $[\Tc],[\Tc_1],[\Tc_2]\in\Subquot(\Ac)$ with $\Tc = \Tc_1 \dotplus \Tc_2$ 
	and $[\Tc]\in\Qc$, then $[\Tc_1],[\Tc_2]\in\Qc$. 
	Then the following assertions are equivalent: 
	\begin{enumerate}[{\rm(i)}] 
		\item\label{SQ7bis_item1} $\Subquot^{\Tr}(\Ac)\cap\Qc=\Subquot^{\Tr}_0(\Ac)\cap\Qc$
		\item\label{SQ7bis_item2} $\Subquot^{\Tr}(\Ac_1)\cap\Qc=\Subquot^{\Tr}_0(\Ac_1)\cap\Qc$ and $\Subquot^{\Tr}(\Ac_2)\cap\Qc=\Subquot^{\Tr}_0(\Ac_2)\cap\Qc$.
	\end{enumerate}
	If in addition $\Ac$ is liminary, then $\Subquot^{\Tr}(\Ac)\cap\Qc=\Subquot^{\Hausd}(\Ac)\cap\Qc$ if and only if 
	$\Subquot^{\Tr}(\Ac_1)\cap\Qc=\Subquot^{\Hausd}(\Ac_1)\cap\Qc$ and $\Subquot^{\Tr}(\Ac_2)\cap\Qc=\Subquot^{\Hausd}(\Ac_2)\cap\Qc$. 
\end{theorem}

\begin{proof}
	The assertion $\eqref{SQ7bis_item1}\Rightarrow\eqref{SQ7bis_item2}$  follows from \cite[Lemma 6.3]{BBL17}.
	
	For $\eqref{SQ7bis_item2}\Rightarrow\eqref{SQ7bis_item1}$, let $[\Tc] \in\Subquot^{\Tr} (\Ac) \cap\Qc$. 
	Then 
	$S= \widehat{\Tc}$ is a locally closed and Hausdorff subset of $X$. 
	Since $\Tc$ has continuous trace it follows from \cite[Cor.~10.5.6]{Dix64} that $\Tc$ is a 
	$C_0(S)$-algebra. 
	
	On the other hand, by Lemma~\ref{SQ5bis}, the sets $S\cap X_1$ and $S\cap X_2$ are both open and closed in $S$. 
	Thus we can write 
	\begin{equation}\label{SQ7-1}
	\Tc = \Tc_1 \dotplus \Tc_2
	\end{equation} 
	where $\Tc_j$ are ideals in $\Tc$, have continuous trace,  and $\widehat{\Tc_j} = S\cap X_j$, $j=1, 2$. 
	
	It follows by the hereditary hypothesis on $\Qc$ that $[\Tc_1],[\Tc_2]\in\Qc$. 
	We also have that $[\Tc_j]\in\Subquot^{\Tr} (\Ac)$: 
	Indeed, since $[\Tc] \in\Subquot^{\Tr} (\Ac)$ there are ideals, $\Ic\subseteq \Jc\subseteq \Ac$ with
	$\Tc = \Jc/\Ic$. 
	Let $p\colon \Jc \to \Jc/\Ic$ be the canonical projection. 
	Then $\Ic \subseteq p^{-1}(\Tc_j)$, $j=1, 2$ are ideals in $\Jc$, hence in $\Ac$, and 
	$p^{-1}(\Tc_j)/\Ic\simeq \Tc_j$. 
	
	For $j= 1, 2$, the sets  $S \cap X_j$ are locally closed in $X_j$, 
	hence there are $ [\Tc_j ']\in  \Subquot^{\Tr}(\Ac_j)= \Subquot^{\Tr}_0(\Ac_j)$, with $\widehat{\Tc_j '} = S\cap X_j$. 
	On the other hand, by \cite[Lemma 6.3]{BBL17}, $[\Tc_j ']\in\Subquot^{\Tr} (\Ac)$ for $j=1, 2$.  
	Since the mapping \eqref{Psi_eq} is injective, we then obtain $\Tc_j'\simeq \Tc_j$, hence $[\Tc_j]\in \Subquot_0^{\Tr} (\Ac) $, $j=1, 2$. 
	It then follows from \eqref{SQ7-1} that $[\Tc]\in \Subquot_0^{\Tr} (\Ac)$. 
	
The proof of the second assertion is completely similar, using again  \cite[Cor.~10.5.6]{Dix64}.
\end{proof}

\subsection*{Limit points of sequences of coadjoint orbits}

\begin{proposition}\label{more1}
	Let $\gg$ be a nilpotent Lie algebra, $\xi\in\gg^*$, 
	and $\{\xi_j\}_{j\in\NN}$ be a sequence in $\gg^*$ with $\lim\limits_{j\to\infty}\xi_j=\xi$. 
	For $j\in\NN$ let $\pg_j\in\Pg(\xi_j)$ such that 
	$\hg:=\lim\limits_{j\to\infty}\pg_j$ exists in $\Gr(\gg)$. 
	Define $P\colon\gg^*\to\hg^*$, $\eta\mapsto\eta\vert_{\hg}$, 
	and  
	$$A_\xi:=\Ad_G^*(G)(P^{-1}(P(\xi))/G\subseteq\gg^*/G.$$ 
	Then $\hg\in\Sg(\xi)$ and  
  $$\Oc_{\xi}\in A_\xi\subseteq L(\{\Oc_{\xi_j}\}_{j\in\NN}).$$
	Moreover, the set $A_\xi$ is pathwise connected. 
\end{proposition}

\begin{proof}
	One has $\hg\in\Sg(\xi)$ by Lemma~\ref{acc_item2}. 
	
	Let us denote $\chi_j:=\exp(\ie\xi_j), \chi:=\exp(\ie\xi)\colon \gg\to\TT$ for every $j\in\NN$. 
	We have 
	$$\Ind_H^G(\chi)={\int\limits_{A_\xi}}^\oplus\pi_{\Oc}^{\oplus n_{\xi,\Oc}}\de\mu(\Oc)$$
	where $\de\mu(\cdot)$ is a measure on $A_\xi$ whose support is equal to $A_\xi$, and $1\le n_{\xi,\Oc}\le\infty$ is a suitable multiplicity 
	and $\pi_{\Oc}$ is a unitary irreducible representation associated by the Kirillov correspondence to any $\Oc\in \gg^*/G$ 
	(see \cite[Thm. on page 552, Thm. 3]{CGG87} and \cite[Thm. 1.1]{Li89}). 
	Then, by \cite[Thm. 1.7]{Fe60}, the representation $\Ind_H^G(\chi)$ of $G$ is weakly equivalent 
	to the family of unitary irreducible representations $\{\pi_{\Oc}\mid \Oc\in A_\xi\}$. 
	
	On the other hand, since $\lim\limits_{j\to\infty}\xi_j=\xi$ in $\gg^*$ 
	and 
	$\hg:=\lim\limits_{j\to\infty}\pg_j$ in $\Gr(\gg)$, 
	it follows that the representation $\Ind_H^G(\chi)$ is weakly contained 
	in the family of irreducible representations $\{\Ind_{P_j}^G(\chi_j)\mid j\in\NN\}$. 
	This can be obtained either by the continuity of induction \cite[Thm. 4.2]{Fe64} or the method of proving that the Kirillov correspondence is continuous,  
	see for instance the proof of \cite[Thm. 8.2]{Ki62}. 
	
	Thus $\{\pi_{\Oc}\mid \Oc\in A_\xi\}$ is weakly contained in $\{\Ind_{P_j}^G(\chi_j)\mid j\in\NN\}$ 
	and then, using the fact that the Kirillov correspondence is a homeomorphism, 
	we obtain $A_\xi\subseteq L(\{\Oc_{\xi_j}\}_{j\in\NN})$, as claimed. 
	
	In order to prove that the set $A_\xi$ is pathwise connected, 
	let $q\colon \gg^*\to\gg^*/G$, $\eta\mapsto\Oc_\eta$, be the canonical quotient map. 
	Since $P\colon\gg^*\to\hg^*$ is a linear map, it follows that for arbitrary 
	$\zeta\in\hg^*$ the set $P^{-1}(\zeta)$ is convex, hence is pathwise connected,  
	and then so is the set $G\times P^{-1}(\zeta)$. 
	Hence the composition of continuous maps 
	$$G\times P^{-1}(\zeta)\mathop{\longrightarrow}\limits^{\Ad^*_G} 
	\Ad_G^*(G)(P^{-1}(\zeta))\mathop{\longrightarrow}\limits^{q} 
	\Ad_G^*(G)(P^{-1}(\zeta))/G$$
	has its image pathwise connected. 
	For $\zeta:=P(\xi)\in\hg^*$ we obtain that $A_\xi$ is pathwise connected, and this completes the proof. 
\end{proof}

\begin{proposition}\label{more2}
	Let $\gg$ be a nilpotent Lie algebra 
	and $\{\Oc_j\}_{j\in\NN}$ be a sequence in~$\gg^*/G$
	with $\dim\Oc_j=2$ for every $j\in\NN$.  
	Then the set $L(\{\Oc_j\}_{j\in\NN})\cap[\gg,\gg]^\perp$ has no isolated points.  
\end{proposition}

\begin{proof}
	By selecting a suitable subsequence of $\{\Oc_j\}_{j\in\NN}$, we may assume 
	that we have $\xi_j\in\Oc_j$ and $\xi\in[\gg,\gg]^\perp$ with $\lim\limits_{j\to\infty}\xi_j=\xi$. 
	Using that $\Gr_{\alg}(\gg)$ is compact (see Remark~\ref{gralg}), we may further select a suitable subsequence 
	and we may thus assume that 
	for every $j\in\NN$ one has $\pg_j\in\Pg(\xi_j)$ for which there exists 
	$\hg:=\lim\limits_{j\to\infty}\pg_j$ in $\Gr(\gg)$. 
	%As $\xi\in[\gg,\gg]^\perp$, one has $\hg\not\in\Pg(\xi)$. 
	Since $2=\dim\Oc_j=2\dim(\gg/\pg_j)$, we obtain $\dim\pg_j=\dim\gg-1$, 
	hence also $\dim\hg=\dim\gg-1$. 
	Then $[\gg,\gg]\subseteq\hg$. 
	Therefore, defining $P\colon\gg^*\to\hg^*$, $\eta\mapsto\eta\vert_{\hg}$, 
	we obtain $P^{-1}(\xi\vert_{\hg})\subseteq[\gg,\gg]^\perp$. 
	This implies that $\Oc_\eta=\{\eta\}$ for every $\eta\in P^{-1}(P(\xi))=\xi+\hg^\perp$ 
	and then the canonical quotient map $q\colon\gg^*\to\gg^*/G$ defines a bijection 
	$$q\vert_{\xi+\hg^\perp}\colon \xi+\hg^\perp\to A_\xi$$ 
	where $A_\xi$ is the set defined in Proposition~\ref{more1}.  
	Since $\dim\hg^\perp=1$, it then follows by Proposition~\ref{more1} that $\Oc_\xi$ is not an isolated point 
	of $L(\{\Oc_j\}_{j\in\NN})\cap[\gg,\gg]^\perp$, and this completes the proof. 
\end{proof}

\begin{remark}\label{more3}
	\normalfont
	Proposition~\ref{more2} is a partial generalization of a property of thread-like groups established in~\cite[Cor. 3.4]{ALS07} 
	(see Example~\ref{ex12}\eqref{ex12_item2} above). 
\end{remark}

\section{Main results}

\subsection*{Continuous-trace subquotients for generalized $ax+b$-groups}  

The next theorem shows that Conjecture~\ref{SQ-prel} holds for nilpotent Lie groups 
of the form $G_D$. 

\begin{theorem}\label{more4}
	If $G$ is 
	a 1-connected nilpotent 
Lie group that has an abelian 
	closed subgroup of codimension~1, 
	then 
	its $C^*$-algebra $\Ac=C^*(G)$ has the property 
	$$\Subquot^{\Tr}(\Ac)%\cap \Subquot^{\Ad^*}(\Ac)
	=\Subquot^{\Tr}_0(\Ac)%\cap \Subquot^{\Ad^*}(\Ac)
	.$$ 
\end{theorem}

\begin{proof}
	Since $G$ has an abelian (connected) closed subgroup of codimension~1, 
	it follows that
	the coadjoint orbits of $G$ have dimensions~$\le 2$.
	Let $X_1$ be the set of $2$-dimensional coadjoint orbits and $X_2$ be the set of $0$-dimensional coadjoint orbits, 
	hence $\widehat{G}=X_1\sqcup X_2$, where $X_1$ is an open subset of $\widehat{G}$. 
	Denote by $\Ac_1$ the ideal of $\Ac$ with $\widehat{\Ac_1}=X_1$. 
	Then the $C^*$-algebra $\Ac_2:=\Ac/\Ac_1$ satisfies $\widehat{\Ac_2}=X_2$ 
	and $0\to\Ac_1 \to\Ac\to\Ac_2\to0$
	is a generalized controlled boundary extension, by Proposition~\ref{more2}. 
	
	The $C^*$-algebra $\Ac_2$ is commutative, and in fact $\Ac_2\simeq\Cc_0([\gg,\gg]^\perp)$, hence clearly 
	$%\Subquot^{\Hausd}(\Ac_2)=
	\Subquot^{\Tr}(\Ac_2)=\Subquot^{\Tr}_0(\Ac_2)
	%\subseteq \Subquot^{\Ad^*}(\Ac)
	$. 
	
	On the other hand, let $[\Tc]\in \Subquot^{\Tr}(\Ac_1)
	%\cap \Subquot^{\Ad^*}(\Ac)
	$. 
	  In particular, $[\Tc]\in \Subquot^{\Tr}(\Ac)$ by \cite[Lemma 6.3]{BBL17}. 
	
	Now let $A\subseteq G$ be the abelian connected closed subgroup of codimension~1. 
	Select $X_0\in\gg\setminus\ag$, so that we have the semidirect product of Lie algebras $\gg=\ag\rtimes\RR X_0$ with its corresponding semidirect product of groups $G=A\rtimes_\alpha \RR$, 
	where, using the group isomorphism $\exp_A\colon (\ag,+)\to A$, we have 
	$$\alpha\colon \RR\to\Aut(A),\quad \alpha_t:=\exp(t(\ad_\gg X_0)\vert_\ag).$$
	We then obtain  $*$-isomorphisms $C^*(G)\simeq C^*(A)\rtimes_\alpha \RR$ 
	and $C^*(A)\simeq\Cc_0(\ag^*)$ via the Fourier transform, 
	hence 
	$$C^*(G)\simeq\Cc_0(\ag^*)\rtimes_{\alpha^*}\RR$$
	where 
	$$\alpha^*\colon \RR\to\Aut(\Cc_0(\ag^*)),\quad 
	\alpha_t^*(\varphi):=\varphi\circ \alpha_t.$$
	Now let us consider the fixed-point set
	$$\ag_0^*:=\{\xi\in\ag^*\mid(\forall t\in\RR)\ \alpha_t^*(\xi)=\xi\}\subseteq\ag^*.$$
	Then $\ag_0^*$ is a linear subspace of $\ag^*$ that is clearly invariant under the action~$\alpha^*$. 
	In particular, its complement $\ag^*\setminus\ag_0^*$ is also $\alpha^*$-invariant and 
	we obtain the commutative diagram 
	$$\xymatrix{0 \ar[r] & \Ac_1\ar[r] \ar[d] & \Ac\ar[r] \ar[d] & \Ac_2\ar[r] \ar[d] & 0 \\
	0\ar[r] & \Cc_0(\ag^*\setminus\ag_0^*)\rtimes_{\alpha^*}\RR\ar[r] &  \Cc_0(\ag^*)\rtimes_{\alpha^*}\RR\ar[r] & \Cc_0(\ag^*_0)\rtimes_{\alpha^*}\RR \ar[r] & 0
	}$$
	whose rows are exact and whose columns are $*$-isomorphisms. 
	Moreover, since the operator $(\ad_\gg X_0)\vert_\ag\colon\ag\to\ag$ has no nonzero purely imaginary eigenvalue, 
	it is easily seen that the restricted action 
	$$\RR\times(\ag^*\setminus\ag_0^*)\to \ag^*\setminus\ag_0^*, 
	(t,\xi)\mapsto\alpha_t^*(\xi)$$
	is free, that is, for any $t\in\RR$ and $\xi\in\ag^*\setminus\ag_0^*$ we have $\alpha_t^*(\xi)=\xi$ if and only if $t=0$. 
	Therefore we obtain $\Subquot^{\Tr}(\Ac_1)=\Subquot^{\Tr}_0(\Ac_1)$ by \cite[Thms. 14 and 17]{Gr77}. 
	Finally, using Theorem~\ref{SQ7bis} for $\Qc:= \Subquot(\Ac)$, we obtain  $%\Subquot^{\Hausd}(\Ac_2)=
	\Subquot^{\Tr}(\Ac)
	%\cap \Subquot^{\Ad^*}(\Ac)
	=\Subquot^{\Tr}_0(\Ac)
	%\cap \Subquot^{\Ad^*}(\Ac)
	$, 
	and this completes the proof. 
\end{proof}

\subsection*{Recovering the dimension from a Lie group $C^*$-algebra}

\begin{proposition}
\label{rec}
Let $G_j$ be a nilpotent Lie group with an abelian 
%connected 
closed subgroup of codimension~1, for $j=1,2$. 
If $C^*(G_1)\simeq C^*(G_2)$, then $\dim\gg_1=\dim\gg_2$ 
and $\dim [\gg_1,\gg_1]=\dim[\gg_2,\gg_2]$. 
\end{proposition}

\begin{proof}
Without loss of generality, we may assume $G_j=\Vc_j\rtimes_{\alpha^{D_j}}\RR$, 
where $\Vc_j$ is a finite-dimensional real vector space and $0\ne D_j\in\End(\Vc_j)$ is a nilpotent operator for $j=1,2$. 
We also denote 
$$n_j:=\dim_\RR\Vc_j\text{ and }k_j:=\dim(\Ker D_j)$$
hence 
\begin{equation}
\label{rec_proof_eq1}
\dim\gg_j=n_j+1\text{ and }\dim[\gg_j,\gg_j]=n_j-k_j
\end{equation}
Since the generic coadjoint orbits of $G_j$ have dimension~2, 
we obtain $\ind G_j=n_j-1$ by \cite[Def. 4.7]{BBL17} 
and then, by \cite[Thm. 4.11]{BBL17}, 
$$\ind C^*(G_j)=n_j-1$$
while 
$$C^*(G_j)/\Comm(C^*(G_j))\simeq \Cc_0([\gg_j,\gg_j]^\perp)\simeq \Cc_0(\RR^{k_j+1}).$$
Now the hypothesis $C^*(G_1)\simeq C^*(G_2)$ implies $n_1=n_2$ and $k_1=k_2$, 
and then \eqref{rec_proof_eq1} implies the assertion. 
\end{proof}

Proposition~\ref{rec} above shows that, in the case of a nilpotent Lie algebra $\gg$ with a 1-codimensional abelian subalgebra, both $\dim\gg$ and $\dim[\gg,\gg]$ can be recovered from~$C^*(G)$.

This result can be regarded as a first step in the study of the $C^*$-rigidity property for this special class of nilpotent Lie groups. 
A further result in that direction is the Proposition~\ref{low}, which uses the notation~$G_D$ from \cite[\S 2]{BB18a}. 

\subsection*{$C^*$-rigidity for low-dimensional coadjoint orbits}

\begin{proposition}
\label{low}
Let $\Vc_0$ be a finite-dimensional real vector space, $D_0\in\End(\Vc_0)$ with  $(D_0)^2=0$, and $G_{D_0}:=\Vc_0\rtimes_{\alpha^{D_0}}\RR$. 
Then the nilpotent Lie group $G_{D_0}$ is $C^*$-rigid within the class of 1-connected nilpotent Lie groups with coadjoint orbits of dimension $\le2$. 
\end{proposition}

\begin{proof}
Let $G$ be a 1-connected nilpotent Lie groups with coadjoint orbits of dimension $\le2$ 
with $C^*(G)\simeq C^*(G_{D_0})$. 
Since $G$ is nilpotent, the condition on coadjoint orbits  implies by \cite[Thm.]{ACL86} that one of the following cases occurs: 
\begin{enumerate}[1.]
	\item $G=G_D=\Vc\rtimes_{\alpha^D}\RR$ for a finite-dimensional real vector space $\Vc$ and a nilpotent operator $D_0\in\End(\Vc_0)$; 
	\item $G=N_{6,15}$, the free 2-step nilpotent Lie group with three generators, so we have $\dim N_{6,15}=6$; 
	\item  $G=N_{5,4}$, where $\dim N_{5,4}=5$. 
\end{enumerate}
Since the groups $N_{5,4}$ and $N_{6,15}$ are $C^*$-rigid within the class of all 1-connected nilpotent Lie groups by \cite[Thms. 1.3 and 6.6]{BB21}, it follows that the above cases 2 and 3 cannot actually occur. 

Thus, we may assume  $G=G_D=\Vc\rtimes_{\alpha^D}\RR$ as in the above case~1. 
Then, by Proposition~\ref{rec}, the assumption $C^*(G_D)=C^*(G)\simeq C^*(G_{D_0})$ 
implies $\dim\Vc=\dim\Vc_0$ and $\dim [\gg_D,\gg_D]=\dim[\gg_{D_0},\gg_{D_0}]$. 
The latter equality shows that 
\begin{equation}
\label{low_proof_eq1}
	\dim(\Ran D)=\dim(\Ran D_0)
\end{equation}
while the equality $\dim\Vc=\dim\Vc_0$ 
shows that we may assume $\Vc=\Vc_0$ without loss of generality. 

If we show that $D^2=0$, then the equality \eqref{low_proof_eq1} implies that the Jordan canonical forms of the nilpotent operators $D$ and $D_0$ coincide, 
hence these operators are similar, and this further implies that the Lie groups $G_{D}$ and $G_{D_0}$ are isomorphic. 

Assuming $D^2\ne0$, we obtain a direct sum decomposition into $D$-invariant linear subspaces $\Vc=\Vc_1\dotplus\Vc_2$ such that the restricted operator $D_2:=D\vert_{\Vc_2}\in\End(\Vc_2)$ has exactly one Jordan cell of size $>2$, 
hence the nilpotent Lie group $G_{D_2}:=\Vc_2\rtimes_{\alpha^{D_2}}\RR$ is a thread-like group  with $\dim G_{D_2}>3$, and then it is known that the complement of the set of characters in the unitary dual $\widehat{G_{D_2}}$ is not  Hausdorff. 
On the other hand, we have the short exact sequence of Lie groups 
\begin{equation*}
	\xymatrix{\Vc_1 \ar@{^{(}->}[r] & G_D \ar@{>>}[r] & G_{D_2}}
\end{equation*}
hence the Lie group isomorphism $G_D/\Vc_1\simeq G_{D_2}$, which leads to a topological embedding of $\widehat{G_{D_2}}$ as a closed subset of $\widehat{G_D}$, 
with the characters of $G_{D_2}$ embedded as characters of $G_D$. 
This implies that the complement of the set of characters in the unitary dual $\widehat{G_D}$ is not Hausdorff. 

 However, the group $G_{D_0}$ is 2-step nilpotent of class $\Tc$ in the sense of \cite[Def. 5.1]{BB21}. 
 Therefore, by the argument in the proof of \cite[Eq. (5.1)]{BB21}, 
 the complement of the set of characters in the unitary dual $\widehat{G_{D_0}}$ is Hausdorff. 
 Since the assumption $C^*(G_D)=C^*(G)\simeq C^*(G_{D_0})$ leads to a homeomorphism $\widehat{G_D}\simeq\widehat{G_{D_0}}$, we thus obtained a contradiction. 
 
 Consequently $D^2=0$, and this completes the proof. 
\end{proof}

\end{document}